\documentclass[10pt]{amsart}
\usepackage{amsmath,amsfonts,amsthm,amssymb,amsxtra,hyperref}
\usepackage{color,graphicx}\parskip 4pt
\newcommand{\Email}[1]{{\sl E-mail address:\/} {\rm\textsf{#1}}}
\newtheorem{thm}{Theorem}
\newtheorem{cor}[thm]{Corollary}
\newtheorem{lem}[thm]{Lemma}
\newtheorem{prop}[thm]{Proposition}
\newtheorem{remark}[thm]{Remark}
\newcommand{\R}{{\mathbb R}}
\renewcommand{\S}{{\mathbb S}}
\newcommand{\N}{{\mathbb N}}
\newcommand{\be}[1]{\begin{equation}\label{#1}}
\newcommand{\ee}{\end{equation}}
\renewcommand{\(}{\left(}
\renewcommand{\)}{\right)}

\renewcommand{\S}{\mathbb{S}}
\newcommand{\iS}[1]{\int_{\S^d}{#1}\;d\mu}

\newcommand{\nrm}[2]{\|{#1}\|_{\mathrm L^{#2}(\S^d)}}

\renewcommand{\H}{\mathrm H}
\newcommand{\Lap}{\Delta_{\S^d}}

\newcommand{\F}[1]{{\mathcal F}[#1]}
\newcommand{\I}[1]{{\mathcal I}[#1]}
\renewcommand{\L}{{\mathcal L}\,}
\newcommand{\ix}[1]{\int_{-1}^1{#1}\;d\nu_d}
\newcommand{\nrmx}[2]{\|#1\|_{#2}}
\newcommand{\scal}[2]{\left\langle{#1},{#2}\right\rangle}

\newcommand{\qp}{p}
\newcommand{\fu}{u}

\usepackage{color}

\newcommand{\nc}{\normalcolor}

\begin{document}
\title[Sharp interpolation inequalities]{Sharp interpolation inequalities on the sphere :\\ new methods and consequences}
\author[J.~Dolbeault , M.J.~Esteban, M.~Kowalczyk, \& M.~Loss]{Jean Dolbeault, Maria J.~Esteban, Michal Kowalczyk, and Michael Loss}
\address{J.~Dolbeault: Ceremade, Universit\'e Paris-Dauphine, Place de Lattre de Tassigny, 75775 Paris C\'edex~16, France.\vspace*{-6pt}
\Email{dolbeaul@ceremade.dauphine.fr}}
\address{M.J.~Esteban: Ceremade, Universit\'e Paris-Dauphine, Place de Lattre de Tassigny, 75775 Paris C\'edex~16, France.\vspace*{-6pt}
\Email{esteban@ceremade.dauphine.fr}}
\address{M.~Kowalczyk: Departamento de Ingenier\'{\i}a  Matem\'atica and Centro de Modelamiento Matem\'atico (UMI 2807 CNRS), Universidad de Chile, Casilla 170 Correo 3, Santiago, Chile.\vspace*{-6pt}
\Email{kowalczy@dim.uchile.cl}}
\address{M.~Loss: Skiles Building, Georgia Institute of Technology, Atlanta GA 30332-0160, USA.
\Email{loss@math.gatech.edu}}
\date{\today}
\begin{abstract} These notes are devoted to various considerations on a family of sharp interpolation inequalities on the sphere, which in dimension two and higher interpolate between Poincar\'e, logarithmic Sobolev and critical Sobolev (Onofri in dimension two) inequalities. We emphasize the connexion between optimal constants and spectral properties of the Laplace-Beltrami operator on the sphere. We shall address a series of related observations and give proofs based on symmetrization and the ultraspherical setting.
\end{abstract}
\maketitle
\thispagestyle{empty}

\section{Introduction}\label{Sec:SphereInterpolation}

The following interpolation inequality holds on the sphere:
\be{Ineq:Interpolation}
\frac{\qp-2}d\iS{|\nabla u|^2}+\iS{|u|^2}\ge\(\iS{|u|^\qp}\)^{2/\qp}\quad\forall\;u\in\H^1(\S^d,d\mu)
\ee
for any $\qp\in(2,2^*]$ with $2^*=2\,d/(d-2)$ if $d\ge 3$ and for any $\qp\in(2,\infty)$ if $d=2$. In~\eqref{Ineq:Interpolation}, $d\mu$ is the uniform probability measure on the $d$-dimensional sphere, that is, the measure induced by Lebesgue's measure on $\S^d\subset\R^{d+1}$, up to a normalization factor such that $\mu(\S^d)=1$.

Such an inequality has been established by M.-F.~Bidaut-V\'eron and L.~V\'eron in \cite{BV-V} in the more general context of compact manifolds with uniformly positive Ricci curvature. Their method is based on the Bochner-Lichnerowicz-Weitzenb\"ock formula and the study of the set of solutions of an elliptic equation which is seen as a bifurcation problem and contains the Euler-Lagrange equation associated to the optimality case in~\eqref{Ineq:Interpolation}. Later, in \cite{MR1230930}, W.~Beckner gave an alternative proof based on Legendre's duality, on the Funk-Hecke formula, which has been proved in \cite{45.0702.01,46.0632.02}, and on the expression of some optimal constants found by E.~Lieb in \cite{MR717827}. D.~Bakry, A.~Bentaleb and S.~Fahlaoui in a series of papers based on the \emph{carr\'e du champ} method and mostly devoted to the ultraspherical operator have shown a result which turns out to give yet another proof, which is anyway very close to the method of \cite{BV-V}. Their computations allow to slightly extend the range of the parameter $\qp$: see \cite{MR1260331,MR2141179,MR1244430,MR1231419,MR1768016,MR1917038,MR1971589,MR2564058,MR2641798}, and \cite{MR674060,MR578933} for earlier related works.

In all computations based on the Bochner-Lichnerowicz-Weitzenb\"ock formula, the choice of exponents in the computations appears somewhat mysterious. The seed for such computations can be found in \cite{MR615628}. Our purpose is on one hand to give alternative proofs, at least for some ranges of the parameter $\qp$, which do not rely on such a very technical choice. On the other hand, we also simplify the existing proofs (see Section \ref{SubSec:classical}).

Inequality~\eqref{Ineq:Interpolation} is remarkable for several reasons:\begin{enumerate}
\item It is optimal in the sense that $1$ is the optimal constant. By H\"older's inequality, we know that $\nrm u2\le\nrm up$ so that the equality case can only be achieved by functions which are constant a.e. Of course, the main issue is to prove that the $(\qp-2)/d$ constant is optimal, which is one of the classical issues of the so-called $A$-$B$ problem, for which we primarily refer to \cite{MR1688256}.
\item If $d\ge3$, the case $\qp=2^*$ corresponds to Sobolev's inequality. Using the stereographic projection as in \cite{MR717827}, we easily recover Sobolev's inequality in the euclidean space $\R^d$ with optimal constant and obtain a simple characterization of the extremal functions found by T. Aubin and G. Talenti: see \cite{MR0448404,MR0289739,MR0463908}.
\item In the limit $\qp\to 2$, one obtains the logarithmic Sobolev inequality on the sphere, while by taking $\qp\to\infty$ if $d=2$, one recovers Onofri's inequality; see \cite{MR1143664} and Corollary~\ref{Cor:Onofri} below.
\end{enumerate}

Exponents are not restricted to $\qp>2$. Consider indeed the functional
\[
\mathcal Q_\qp[u]:=\frac{\qp-2}d\,\frac{\iS{|\nabla u|^2}}{\(\iS{|u|^\qp}\)^{2/\qp}-\iS{|u|^2}}
\]
for $\qp\in[1,2)\cup(2,2^*]$ if $d\ge 3$ or $\qp\in[1,2)\cup(2,\infty)$ if $d=2$, and
\[
\mathcal Q_2[u]:=\frac1d\,\frac{\iS{|\nabla u|^2}}{\iS{|u|^2\log\(|u|^2/\iS{|u|^2}\)}}
\]
for any $d\ge1$. Because $d\mu$ is a probability measure, $\(\iS{|u|^\qp}\)^{2/\qp}-\iS{|u|^2}$ is nonnegative if $\qp>2$, nonpositive if $\qp\in[1,2)$, and equal to zero if and only if $u$ is constant a.e. Denote by $\mathcal A$ the set of $\H^1(\S^d,d\mu)$ functions which are not a.e.~constant and consider the infimum
\be{Eqn:Infimum}
\mathcal I_\qp:=\inf_{u\in\mathcal A}\mathcal Q_\qp[u]\;.
\ee
With these notations, we can state a slightly more general result than the one of~\eqref{Ineq:Interpolation}, which goes as follows and also covers the range $\qp\in[1,2]$.
\begin{thm}\label{Thm:Main} With the above notations, $\mathcal I_\qp=1$ for any $\qp\in[1,2^*]$ if $d\ge 3$, or any $\qp\in[1,\infty)$ if $d=1$, $2$.\end{thm}
As already explained above, in the case $(2, 2^*]$ the above theorem was proved first in \cite[Corollary 6.2]{BV-V}, and then in \cite{MR1230930} using previous results by E.~Lieb in \cite{MR717827} and the Funk-Hecke formula (see \cite{45.0702.01,46.0632.02}). The case $p=2$ was covered in \cite{MR1230930}. The whole range $p\in[1, 2^*]$ was covered in the case of the ultraspherical operator in \cite{MR2564058,MR2641798}. Here we give alternative proofs for various ranges of $p$, which are less technical, and interesting by themselves, as well as some extensions.

Notice that the case $\qp=1$ can be written as
\[
\iS{|\nabla u|^2}\ge d\left[\iS{|u|^2}-\(\iS{|u|}\)^2\right]\quad\forall\;u\in\H^1(\S^d,d\mu)\;,
\]
which is equivalent to the usual Poincar\'e inequality
\[
\iS{|\nabla u|^2}\ge d\iS{|u-\overline u|^2}\quad\forall\;u\in\H^1(\S^d,d\mu)\quad\mbox{with}\quad\overline u=\iS u\;.
\]
See Remark~\ref{Rem:Poincare}, for more details. The case $\qp=2$ provides the logarithmic Sobolev inequality on the sphere. It holds as consequence of the inequality for $\qp\neq 2$ (see Corollary~\ref{Cor:LogSob}).

For $\qp\neq 2$, the existence of a minimizer of
\[
u\mapsto\iS{|\nabla u|^2}+\frac{d\,\mathcal I_p}{\qp-2}\left[\nrm u2^2-\nrm up^2\right]
\]
in $\big\{u\in\H^1(\S^d,d\mu)\,:\,\iS{|u|^\qp}=1\big\}$ is easily achieved by variational methods and will be taken for granted. Compactness for either $\qp\in[1,2)$ or $2<\qp<2^*$ is indeed classical, while the case $\qp=2^*$, $d\ge 3$ can be studied by concentration-compactness methods. If a function $u\in\H^1(\S^d,d\mu)$ is optimal for \eqref{Ineq:Interpolation} with $p\neq2$, then it is solves the Euler-Lagrange equation
\be{Eqn:EL}
-\Lap u=\frac{d\,\mathcal I_p}{\qp-2}\,\left[\nrm up^{2-p}\,u^{\qp-1}-u\right]
\ee
where $\Lap$ denotes the Laplace-Beltrami operator on the sphere $\S^d$.

In any case, it is possible to normalize the $\mathrm L^\qp(\S^d)$-norm of $u$ to $1$ without restriction because of the zero homogeneity of $\mathcal Q_\qp$. It turns out that the optimality case is achieved by the constant function, with value $u\equiv1$ if we assume $\iS{|u|^\qp}=1$, in which case the inequality degenerates because both sides are equal to $0$. This explains why the dimension $d$ shows up here: the sequence $(u_n)_{n\in\N}$ such that
\[
u_n(x)=1+\frac 1n\,v(x)
\]
with $v\in\H^1(\S^d,d\mu)$ such that $\iS v=0$ is indeed minimizing if and only if
\[
\iS{|\nabla v|^2}\ge d\iS{|v|^2}\;,
\]
and the equality case is achieved if $v$ is an optimal function for the above Poincar\'e inequality, \emph{i.e.}~a function associated to the first non-zero eigenvalue of the Laplace-Beltrami operator $-\,\Lap$ on the sphere $\S^d$. Up to a rotation, this means
\[
v(\xi)=\xi_d\quad\forall\;\xi=(\xi_0,\,\xi_1,\ldots\xi_d)\in\S^d\subset\R^{d+1}
\]
since $-\,\Lap v=d\,v$. Recall that the corresponding eigenspace of $-\,\Lap$ is $d$ dimensional and generated by the composition of $v$ with an arbitrary rotation.

\subsection{The logarithmic Sobolev inequality}\label{Sec:LogSob}

As a first classical consequence of \eqref{Eqn:Infimum}, we have a logarithmic Sobolev inequality. This result is rather classical; related forms of the result can be found for instance in~\cite{Bakry-Emery85} or in \cite{AMTU}.
\begin{cor}\label{Cor:LogSob} Let $d\ge1$. For any $u\in\H^1(\S^d,d\mu)\setminus\{0\}$, we have
\[
\iS{|u|^2\,\log\(\frac{|u|^2}{\iS{|u|^2}}\)}\le\frac 2d\iS{|\nabla u|^2}\;.
\]
Moreover, the constant $\frac 2d$ is sharp.\end{cor}
\begin{proof} The inequality is achieved by taking the limit as $\qp\to 2$ in \eqref{Eqn:Infimum}. To see that the constant $\frac 2d$ is sharp, we can observe that
\[
\lim_{\varepsilon\to 0}\iS{|1+\varepsilon\,v|^2\,\log\(\frac{|1+\varepsilon\,v|^2}{\iS{|1+\varepsilon\,v|^2}}\)}=2\iS{\left|v-\overline v\right|^2}
\]
with $\overline v=\iS v$. The result follows by taking $v(\xi)=\xi_d$.\end{proof}

\section{Extensions}\label{Sec:Extensions}

\subsection{Onofri's inequality}\label{Sec:Onofri}

In case of dimension $d=2$, \eqref {Ineq:Interpolation} holds for any $\qp>2$ and we recover Onofri's inequality by taking the limit $\qp\to\infty$. This result is standard in the literature: see for instance \cite{MR1230930}. For completeness, let us give a statement and a short proof.
\begin{cor}\label{Cor:Onofri} Let $d=1$ or $d=2$. For any $v\in\H^1(\S^d,d\mu)$, we have
\[
\iS{e^{v-\overline v}}\le e^{\frac 1{2\,d}\iS{|\nabla v|^2}}
\]
where $\overline v=\iS v$ is the average of $v$. Moreover, the constant $\frac 1{2\,d}$ in the right hand side is sharp.\end{cor}
\begin{proof}In dimension $d=1$ or $d=2$, Inequality~\eqref{Ineq:Interpolation} holds for any $\qp>2$. Take $u=1+v/\qp$ and consider the limit as $\qp\to\infty$. We observe that
\[
\iS{|\nabla u|^2}=\frac 1{\qp^2}\iS{|\nabla v|^2}\quad\mbox{and}\quad\lim_{\qp\to\infty}\iS{|u|^\qp}=\iS{e^v}
\]
so that
\[
\(\iS{|u|^\qp}\)^{2/\qp}-1\sim \frac 2\qp\,\log\(\iS{e^v}\)\quad\mbox{and}\quad\iS{|u|^2}-1\sim \frac 2\qp\,\iS v\;.
\]
The conclusion holds by passing to the limit $\qp\to\infty$ in Inequality~\eqref{Ineq:Interpolation}. Optimality is once more achieved by considering $v=\varepsilon\,v_1$, $v_1(\xi)=\xi_d$, $d=1$ and Taylor expanding both sides of the inequality in terms of $\varepsilon>0$, small. Notice indeed that $-\Lap v_1=\lambda_1\,v_1$ with $\lambda_1=d$, so that
\[
\nrm{\nabla u}2^2=\varepsilon^2\,\nrm{\nabla v_1}2^2=\varepsilon^2\,d\,\nrm{v_1}2^2\;,
\]
$\iS{v_1}=\overline v_1=0$, and
\[
\iS{e^{v-\overline v}}-1\sim\frac{\varepsilon^2}2\iS{|v-\overline v|^2}=\frac 12\,\varepsilon^2\,\nrm{v_1}2^2\;.
\]
\end{proof}

\subsection{Interpolation and a spectral approach for \texorpdfstring{$\qp\in(1,2)$}{qp in (1,2)}}\label{SubSec:OneTwo}

In \cite{MR954373}, W.~Beckner gave a method to prove interpolation inequalities between logarithmic Sobolev and Poincar\'e inequalities in case of a Gaussian measure. Here we shall prove that the method extends to the case of the sphere and therefore provides another family of interpolating inequalities, in a new range: $\qp\in[1,2)$, again with optimal constants. For further considerations on inequalities that interpolate between Poincar\'e and logarithmic Sobolev inequalities, we refer to \cite{MR2152502,ABD05,ABD,Bakry-Emery85,MR772092,MR2766956,MR2609029,MR2081075,MR1796718} and references therein.

Our purpose is to extend \eqref{Ineq:Interpolation} written as
\be{Eqn:BecknerExtended}
\frac 1d\iS{|\nabla u|^2}\ge \frac{\(\iS{|u|^\qp}\)^{2/\qp}-\iS{|u|^2}}{\qp-2}\quad\forall\;u\in\H^1(\S^d,d\mu)
\ee
to the case $\qp\in[1,2)$. Let us start with a remark.
\begin{remark}\label{Rem:Poincare} At least for any nonnegative function $v$, using the fact that $\mu$ is a probability measure on $\S^d$, we may notice that
\[
\iS{\left|v-\overline v\right|^2}=\iS{\left|v\right|^2}-\(\iS v\)^2
\]
can be rewritten as
\[
\iS{\left|v-\overline v\right|^2}=\frac{\iS{|v|^2}-\(\iS{|v|^\qp}\)^{2/\qp}}{2-\qp}\;,
\]
for $\qp=1$, hence extending \eqref{Ineq:Interpolation} to the case $q=1$. However, as already noticed for instance in \cite{ABD05}, the inequality
\[
\iS{|v|^2}-\(\iS{|v|}\)^2\le \frac 1d\iS{|\nabla v|^2}
\]
also means that, for any $c\in\R$,
\[
\iS{|v+c|^2}-\(\iS{|v+c|}\)^2\le \frac 1d\iS{|\nabla v|^2}\;.
\]
If $v$ is bounded from below a.e.~with respect to $\mu$ and $c>-\mathrm{infess}_\mu v$, so that $v+c>0$ $\mu$ a.e., the left hand-side is
\[
\iS{|v+c|^2}-\(\iS{|v+c|}\)^2=c^2+2\,c\iS v+\iS{|v|^2}-\(c+\iS v\)^2=\iS{\left|v-\overline v\right|^2}\;,
\]
so that the inequality is the usual Poincar\'e inequality. By density, we recover that \eqref{Eqn:BecknerExtended} written for $\qp=1$ exactly amounts to Poincar\'e's inequality written not only for $|v|$, but also for any $v\in\H^1(\S^d,d\mu)$.
\end{remark}

\medskip Next, using the method introduced by W.~Beckner in \cite{MR954373} in case of a Gaussian measure, we are in position to prove \eqref{Eqn:BecknerExtended} for any $\qp\in(1,2)$, knowing that the inequality holds for $\qp=1$ and $\qp=2$.
\begin{prop}\label{Prop:4} Inequality \eqref{Eqn:BecknerExtended} holds for any $\qp\in(1,2)$ and any $d\ge 1$. Moreover $d$ is the optimal constant. \end{prop}
\begin{proof} Optimality can be checked by Taylor expanding $u=1+\varepsilon\,v$ at order two in terms of $\varepsilon>0$ as in the case $\qp=2$ (logarithmic Sobolev inequality). To establish the inequality itself, we may proceed in two steps.

\medskip\noindent\emph{$1^{\rm st}$ step: Nelson's hypercontractivity result.\/} Although the result can be established by direct methods, we follow here the strategy of Gross in \cite{Gross75}, which proves the equivalence of the optimal hypercontractivity result and the optimal logarithmic Sobolev inequality.

Consider the heat equation of $\S^d$, namely
\[
\frac{\partial f}{\partial t}=\Lap f
\]
with initial datum $f(t=0,\cdot)=u\in L^{2/\qp}(\S^d)$, for some $\qp\in(1,2]$, and let $F(t):=\nrm {f(t,\cdot)}{p(t)}$. The key computation goes as follows.
\begin{multline*}
\frac{F'}F=\frac d{dt}\,\log F(t)=\frac d{dt}\,\left[\frac 1{p(t)}\,\log\(\iS{|f(t,\cdot)|^{p(t)}}\)\right]\\
=\frac{p'}{p^2\,F^p}\left[\iS{v^2\log\(\frac{v^2}{\iS{v^2}}\)}+4\,\frac{p-1}{p'}\,\iS{|\nabla v|^2}\right]
\end{multline*}
with $v:=|f|^{p(t)/2}$. Assuming that $4\,\frac{p-1}{p'}=\frac 2d$, that is
\[
\frac{p'}{p-1}=2\,d\;,
\]
we find that
\[
\log\(\frac{p(t)-1}{\qp-1}\)=2\,d\,t
\]
if we require that $p(0)=\qp<2$. Let $t_*>0$ be such that $p(t_*)=2$. As a consequence of the above computation, we have
\be{Ineq:Nelson}
\nrm{f(t_*,\cdot)}2\le\nrm u{2/\qp}\quad\mbox{if}\quad\frac 1{\qp-1}=e^{2\,d\,t_*}\;.
\ee

\medskip\noindent\emph{$2^{\rm nd}$ step: Spectral decomposition.\/} Let $u=\sum_{k\in\N}u_k$ be a decomposition of the initial datum on the eigenspaces of $-\Lap$ and denote by $\lambda_k=k\,(d+k-1)$ the ordered sequence of the eigenvalues: $-\Lap u_k=\lambda_k\,u_k$ (see for instance \cite{MR0282313}). Let $a_k=\nrm{u_k}2^2$. As a straightforward consequence of this decomposition, we know that $\nrm u2^2=\sum_{k\in\N}a_k$, $\nrm{\nabla u}2^2=\sum_{k\in\N}\lambda_k\,a_k$,
\[
\nrm{f(t_*,\cdot)}2^2=\sum_{k\in\N}a_k\,e^{-2\,\lambda_k\,t_*}\;.
\]
Using \eqref{Ineq:Nelson}, it follows that
\[
\frac{\(\iS{|u|^\qp}\)^{2/\qp}-\iS{|u|^2}}{\qp-2}\le\frac{\(\iS{|u|^2}\)-\iS{|f(t_*,\cdot)|^2}}{2-\qp}=\frac 1{2-\qp}\sum_{k\in\N^*}\lambda_k\,a_k\,\frac{1-e^{-2\,\lambda_k\,t_*}}{\lambda_k}\;.
\]
Notice that $\lambda_0=0$ so that the term corresponding to $k=0$ can be omitted in the series. Since $\lambda\mapsto\frac{1-e^{-2\,\lambda\,t_*}}{\lambda}$ is decreasing, we can bound $\frac{1-e^{-2\,\lambda_k\,t_*}}{\lambda_k}$ from above by $\frac{1-e^{-2\,\lambda_1\,t_*}}{\lambda_1}$ for any $k\ge 1$. This proves that
\[
\frac{\(\iS{|u|^\qp}\)^{2/\qp}-\iS{|u|^2}}{\qp-2}\le\frac{1-e^{-2\,\lambda_1\,t_*}}{(2-\qp)\,\lambda_1}\sum_{k\in\N^*}\lambda_k\,a_k=\frac{1-e^{-2\,\lambda_1\,t_*}}{(2-\qp)\,\lambda_1}\,\nrm{\nabla u}2^2\;.
\]
The conclusion easily follows if we notice that $\lambda_1=d$, and $e^{-2\,\lambda_1\,t_*}=\qp-1$ so that
\[
\frac{1-e^{-2\,\lambda_1\,t_*}}{(2-\qp)\,\lambda_1}=\frac 1d\;.
\]
The optimality of this constant can be checked as in the case $\qp>2$ by a Taylor expansion of $u=1+\varepsilon\,v$ at order two in terms of $\varepsilon>0$, small.
\end{proof}

\section{Symmetrization and the ultraspherical framework}

\subsection{A reduction to the ultraspherical framework}\label{SubSec:reduction}

\medskip We denote by $(\xi_0,\,\xi_1,\ldots\xi_d)$ the coordinates of an arbitrary point $\xi\in\S^d$, with $\sum_{i=0}^d|\xi_i|^2=1$. The following symmetry result is kind of folklore in the literature and we can quote \cite{MR0402083,MR717827,MR1164616} for various related results.
\begin{lem}\label{Lem:Sym} Up to a rotation, any minimizer of \eqref{Eqn:Infimum} depends only on $\xi_d$.\end{lem}
\begin{proof} Let $u$ be a minimizer for $\mathcal Q_\qp$. By writing $u$ in \eqref{Ineq:Interpolation} in spherical coordinates $\theta\in[0,\pi]$, $\varphi_1$, $\varphi_2$,... $\varphi_{d-1}\in[0,2\pi)$ and using decreasing rearrangements (see for instance~\cite{brock2000general}), it is not difficult to prove that among optimal functions, there is one which depends only on $\theta$. Moreover, equality in the rearrangement inequality means that $u$ has to depend on only one coordinate, $\xi_d=\sin\theta$.\end{proof} 

Let us observe that the problem on the sphere can be reduced to a problem involving the ultraspherical operator:
\begin{enumerate}
\item[$\bullet$] Using Lemma~\ref{Lem:Sym}, we know that \eqref{Ineq:Interpolation} is equivalent~to
\[
\frac{\qp-2}d\int_0^\pi|v'(\theta)|^2\;d\sigma+\int_0^\pi|v(\theta)|^2\;d\sigma\ge\(\int_0^\pi |v(\theta)|^\qp\;d\sigma\)^\frac 2\qp
\]
for any function $v\in\mathrm H^1([0,\pi],d\sigma)$, where
\[
d\sigma(\theta):=\frac{(\sin\theta)^{d-1}}{Z_d}\,d\theta\quad\mbox{with}\quad Z_d:=\sqrt\pi\,\frac{\Gamma(\tfrac d2)}{\Gamma(\tfrac{d+1}2)}\;.
\]
\item[$\bullet$] The change of variables $x=\cos\theta$, $v(\theta)=f(x)$ allows to rewrite the inequality as
\[\label{Ultraspherical}
\frac{\qp-2}d\int_{-1}^1|f'|^2\;\nu\;d\nu_d+\ix{|f|^2}\ge\(\ix{|f|^\qp}\)^\frac 2\qp
\]
where $d\nu_d$ is the probability measure defined by
\[
\nu_d(x)\,dx=d\nu_d(x):=Z_d^{-1}\,\nu^{\frac d2-1}\,dx\quad\mbox{with}\quad\nu(x):=1-x^2\;,\quad Z_d=\sqrt\pi\,\frac{\Gamma(\tfrac d2)}{\Gamma(\tfrac{d+1}2)}\;.
\]
\end{enumerate}
We may also want to prove the result in case $p<2$, to have the counterpart of Theorem~\ref{Thm:Main} in the ultraspherical setting. On $[-1,1]$, consider the probability measure $d\nu_d$ and define
\[
\nu(x):=1-x^2\,,
\]
so that $d\nu_d=Z_d^{-1}\,\nu^{\frac d2-1}\,dx$. We consider the space $\mathrm L^2((-1,1),d\nu_d)$ with scalar product
\[
\scal{f_1}{f_2}=\ix{f_1\,f_2}
\]
and use the notation
\[
\nrmx f\qp=\(\ix{f^\qp}\)^\frac 1\qp\;.
\]
On $\mathrm L^2((-1,1),d\nu_d)$, we define the self-adjoint \emph{ultraspherical} operator by
\[
\L f:=(1-x^2)\,f''-d\,x\,f'=\nu\,f''+\frac d2\,\nu'\,f'
\]
which satisfies the identity
\[
\scal{f_1}{\L f_2}=-\ix{f_1'\,f_2'\;\nu}\;.
\]
Then the result goes as follows.
\begin{prop}\label{Prop:Ultra} Let $p\in[1,2^*]$, $d\ge 1$. Then we have
\be{InterpUS}
-\scal f{\L f}=\ix{|f'|^2\;\nu}\ge d\,\frac{\nrmx f\qp^2-\nrmx f2^2}{\qp-2}\quad\forall\,f\in\H^1([-1,1],d\nu_d)
\ee
if $p\neq 2$, and
\[
-\scal f{\L f}=\frac d2\ix{|f|^2\,\log\(\frac{|f|^2}{\nrmx f2^2}\)}
\]
if $p=2$.
\end{prop}
We may notice that the proof in~\cite{BV-V} requires $d\ge2$ while the case $d=1$ is also covered in \cite{MR1230930}. In Bentaleb \emph{et al.,} the restriction $d\ge2$ has been removed in~\cite{MR2641798}. Our proof is inspired by \cite{BV-V} and \cite{MR1231419,MR1971589}, but it is a simplification (in the particular case of the ultraspherical operator) in the sense that only integration by parts and elementary estimates are used.

\subsection{A proof of Proposition~\ref{Prop:Ultra}}\label{SubSec:classical}

Let us start with some preliminary observations. The operator $\L$ does not commute with the derivation, but we have the relation
\[\label{Commutation}
\left[\frac\partial{\partial x},\L\right]\,\fu=\(\L \fu\)'-\L \fu'=-2\,x\,\fu''-d\,\fu'\;.
\]
As a consequence, we obtain
\[
\scal{\L \fu}{\L \fu}=-\ix{\fu'\,\(\L \fu\)'\;\nu}=-\ix{\fu'\,\L \fu'\;\nu}+\ix{\fu'\,\(2\,x\,\fu''+d\,\fu'\)\;\nu}
\]
and
\[
\scal{\L \fu}{\L \fu}=\ix{|\fu''|^2\;\nu^2}-d\,\scal \fu{\L \fu}\;,
\]
\be{Gamma2}
\ix{(\L u)^2}=\scal{\L \fu}{\L \fu}=\ix{|\fu''|^2\;\nu^2}+d\ix{|\fu'|^2\;\nu}\;.
\ee
On the other hand, a few integrations by parts show that
\be{L-Gamma}
\scal{\frac{|\fu'|^2}\fu\;\nu}{\L \fu}=\frac d{d+2}\ix{\frac{|\fu'|^4}{\fu^2}\;\nu^2}-\,2\,\frac{d-1}{d+2}\ix{\frac{|\fu'|^2\,\fu''}\fu\;\nu^2}\;,
\ee
where we have used the fact that $\nu\,\nu'\,\nu_d=\frac 2{d+2}\,(\nu^2\,\nu_d)'$.

\medskip Let $p\in(1,2)\cup (2,2^*)$. In $H^1([-1,1],d\nu_d)$, consider now a minimizer $f$ for the functional
\[
f\mapsto \ix{|f'|^2\;\nu}- d\,\frac{\nrmx f\qp^2-\nrmx f2^2}{\qp-2}=:\mathcal F[f]
\]
made of the difference of the two sides in inequality \eqref{InterpUS}. The existence of such a minimizer can be proved by classical minimization and compactness arguments. Up to a multiplication by a constant, $f$ satisfies the Euler-Lagrange equation
\[
-\frac{p-2}d\,\L f+f=f^{p-1}\;.
\]
Let $\beta$ be a real number to be fixed later and define $u$ such that $f=u^\beta$, so that
\[
\L f=\beta\,u^{\beta-1}\,\(\L u+(\beta-1)\,\frac{|u'|^2}u\,\nu\)\,.
\]
Then $u$ is a solution to
\[
-\L u-(\beta-1)\,\frac{|u'|^2}u\,\nu+\lambda\,u=\lambda\,u^{1+\beta\,(p-2)}\quad\mbox{with}\quad\lambda:=\frac d{(p-2)\,\beta}\;.
\]
If we multiply the equation for $u$ by $\frac{|u'|^2}u\,\nu$ and integrate, we get
\[
-\ix{\L u\,\frac{|u'|^2}u\,\nu}-(\beta-1)\ix{\frac{|u'|^4}{u^2}\,\nu^2}+\lambda\ix{|u'|^{2}\;\nu}=\lambda\,\ix{u^{\beta\,(p-2)}\,|u'|^2\;\nu}\;.
\]
If we multiply the equation for $u$ by $-\,\L u$ and integrate, we get
\[
\ix{(\L u)^2}+(\beta-1)\ix{\L u\,\frac{|u'|^2}u\;\nu}+\lambda\ix{|u'|^2\;\nu}=(\lambda+d)\ix{u^{\beta\,(p-2)}\,|u'|^2\;\nu}\;.
\]
Collecting terms, we have found that
\[
\ix{(\L u)^2}+\(\beta+\frac d\lambda\)\ix{\L u\,\frac{|u'|^2}u\;\nu}+(\beta-1)\,\(1+\frac d\lambda\)\ix{\frac{|u'|^4}{u^2}\,\nu^2}-d\ix{|u'|^2\;\nu}=0\,.
\]
Using~\eqref{Gamma2} and~\eqref{L-Gamma}, we get
\begin{multline*}
\ix{|\fu''|^2\;\nu^2}+\(\beta+\frac d\lambda\)\left[\frac d{d+2}\ix{\frac{|\fu'|^4}{\fu^2}\;\nu^2}-\,2\,\frac{d-1}{d+2}\ix{\frac{|\fu'|^2\,\fu''}\fu\;\nu^2}\right]\\
+(\beta-1)\,\(1+\frac d\lambda\)\ix{\frac{|u'|^4}{u^2}\,\nu^2}=0 \,,
\end{multline*}
that is
\be{identity0}
\mathsf a\ix{|\fu''|^2\;\nu^2}+2\,\mathsf b\ix{\frac{|\fu'|^2\,\fu''}\fu\;\nu^2}+\mathsf c\ix{\frac{|\fu'|^4}{\fu^2}\;\nu^2}=0
\ee
where
\begin{eqnarray*}
&&\mathsf a=1\;,\\
&&\mathsf b=-\,\(\beta+\frac d\lambda\)\,\frac{d-1}{d+2}\;,\\
&&\mathsf c=\(\beta+\frac d\lambda\)\frac d{d+2}+(\beta-1)\,\(1+\frac d\lambda\)\;.
\end{eqnarray*}
Using $\frac d\lambda=(p-2)\,\beta$, we observe that the reduced discriminant
\[
\delta=\mathsf b^2-\mathsf a\,\mathsf c<0
\]
can be written as
\[
\delta=A\,\beta^2+B\,\beta+1\quad\mbox{with}\quad A=(p-1)^2\,\frac{(d-1)^2}{(d+2)^2}-p+2\quad\mbox{and}\quad B=p-3-\frac{d\,(p-1)}{d+2}\;.
\]
If $p<2^*$, $B^2-4\,A$ is positive and it is therefore possible to find $\beta$ such that $\delta<0$. 

Hence, if $p<2^*$, we have shown that $\mathcal F[f]$ is positive unless the three integrals \eqref{identity0} are equal to $0$, that is, $u$ is constant. It follows that $\mathcal F[f]=0$, which proves \eqref{InterpUS} if $p\in(1,2)\cup (2,2^*)$. The cases $p=1$, $p=2$ (\emph{cf.} Corollary~\ref{Cor:LogSob}) and $p=2^*$ can be proved as limit cases. This concludes the proof of Proposition~\ref{Prop:Ultra}.

\section{A proof based on a flow in the ultraspherical setting}

Inequality~\eqref{InterpUS} can be rewritten for $g=f^\qp$, \emph{i.e.}~$f=g^\alpha$ with $\alpha=1/\qp$, as
\[
-\scal f{\L f}=-\scal{g^\alpha}{\L g^\alpha}=:\I g\ge d\,\frac{\nrmx g1^{2\,\alpha}-\nrmx{g^{2\,\alpha}}1}{\qp-2}=:\F g
\]

\subsection{Flow}
Consider the flow associated to $\L$, that is
\be{flow}
\frac{\partial g}{\partial t}=\L g\;,
\ee
and observe that
\[
\frac d{dt}\,\nrmx g1=0\;,\quad
\frac d{dt}\,\nrmx{g^{2\,\alpha}}1=-\,2\,(\qp-2)\,\scal f{\L f}=2\,(\qp-2)\,\ix{|f'|^2\;\nu}
\]
which finally gives
\[
\frac d{dt}\F{g(t,\cdot)}=-\frac d{\qp-2}\,\frac d{dt}\,\nrmx{g^{2\,\alpha}}1=-\,2\,d\,\I{g(t,\cdot)}
\]

\subsection{Method}
If \eqref{InterpUS} holds, then
\be{EntropyUS}
\frac d{dt}\F{g(t,\cdot)}\le-\,2\,d\,\F{g(t,\cdot)}\;,
\ee
thus proving
\[
\F{g(t,\cdot)}\le\F{g(0,\cdot)}\,e^{-\,2\,d\,t}\quad\forall\;t\ge0\;.
\]
This estimate is actually equivalent to \eqref{InterpUS} as can be shown by estimating $\frac d{dt}\F{g(t,\cdot)}$ at $t=0$.

The method based on the Bakry-Emery approach amounts to establish first that
\be{BE}
\frac d{dt}\I{g(t,\cdot)}\le-\,2\,d\,\I{g(t,\cdot)}
\ee
and prove \eqref{EntropyUS} by integrating the estimate on $t\in[0,\infty)$: since
\[
\frac d{dt}\(\F{g(t,\cdot)}-\I{g(t,\cdot)}\)\ge0
\]
and $\lim_{t\to\infty}\(\F{g(t,\cdot)}-\I{g(t,\cdot)}\)=0$, this means that
\[
\F{g(t,\cdot)}-\I{g(t,\cdot)}\le0\quad\forall\;t\ge0
\]
which is precisely \eqref{InterpUS} written for $f(t,\cdot)$ for any $t\ge0$ and in particular for any initial value $f(0,\cdot)$.

The equation for $g=f^\qp$ can be rewritten in terms of $f$ as
\[
\frac{\partial f}{\partial t}=\L f+(\qp-1)\,\frac{|f'|^2}f\;\nu\;.
\]
Hence we have
\[
-\frac 12\,\frac d{dt}\ix{|f'|^2\;\nu}=\frac 12\,\frac d{dt}\,\scal f{\L f}=\scal{\L f}{\L f}+(\qp-1)\,\scal{\frac{|f'|^2}f\;\nu}{\L f}\;,
\]

\subsection{An inequality for the Fisher information} Instead of proving~\eqref{InterpUS}, we will established the following stronger inequality. For any $\qp\in(2,2^\sharp]$,
\be{InterpFisherUS}
\scal{\L f}{\L f}+(\qp-1)\,\scal{\frac{|f'|^2}f\;\nu}{\L f}+d\,\scal f{\L f}\ge0\;.
\ee
Notice that~\eqref{InterpUS} holds under the restriction $\qp\in(2,2^\sharp]$, which is stronger than $\qp\in(2,2^*]$. We do not know whether the exponent $2^\sharp$ in \eqref{InterpFisherUS} is sharp or not.

\subsection{Proof of \texorpdfstring{\eqref{InterpFisherUS}}{12}}

Using~\eqref{Gamma2} and~\eqref{L-Gamma} with $\fu=f$, we find that
\begin{multline*}
\frac d{dt}\ix{|f'|^2\;\nu}+\,2\,d\ix{|f'|^2\;\nu}\\
=-\,2\,\ix{\(|f''|^2+(\qp-1)\,\frac d{d+2}\frac{|f'|^4}{f^2}-\,2\,(\qp-1)\,\frac{d-1}{d+2}\frac{|f'|^2\,f''}f\)\,\nu^2}\;.
\end{multline*}
The right hand side is nonpositive if
\[
|f''|^2+(\qp-1)\,\frac d{d+2}\frac{|f'|^4}{f^2}-\,2\,(\qp-1)\,\frac{d-1}{d+2}\frac{|f'|^2\,f''}f
\]
is pointwise nonnegative, which is granted if
\[
\left[(\qp-1)\,\frac{d-1}{d+2}\right]^2\le(\qp-1)\,\frac d{d+2}\;,
\]
a condition which is exactly equivalent to $\qp\le2^\sharp$.

\subsection{An improved inequality} For any $\qp\in(2,2^\sharp)$, we can write that
\begin{multline*}
|f''|^2+(\qp-1)\,\frac d{d+2}\frac{|f'|^4}{f^2}-\,2\,(\qp-1)\,\frac{d-1}{d+2}\frac{|f'|^2\,f''}f\\
=\alpha\,|f''|^2+\frac{\qp-1}{d+2}\left|\frac{d-1}{\sqrt d}\,f''-\sqrt d\,\frac{|f'|^2}f\right|^2\ge \alpha\,|f''|^2
\end{multline*}
where
\[
\alpha:=1-(\qp-1)\,\frac{(d-1)^2}{d\,(d+2)}
\]
is positive. Now, using the Poincar\'e inequality
\[
\int_{-1}^1|f''|^2\;d\nu_{d+4}\ge(d+2)\int_{-1}^1|f'-\overline{f'}|^2\;d\nu_{d+2}
\]
where
\[
\overline{f'}:=\int_{-1}^1f'\;d\nu_{d+2}=-d\ix{x\,f}\;,
\]
we obtain an improved form of \eqref{InterpFisherUS}, namely
\[
\scal{\L f}{\L f}+(\qp-1)\,\scal{\frac{|f'|^2}f\;\nu}{\L f}+[d+\alpha\,(d+2)]\,\scal f{\L f}\ge0\;,
\]
if we can guarantee that $\overline{f'}\equiv0$ along the evolution determined by \eqref{flow}. This is the case if assume that $f(x)=f(-x)$ for any $x\in[-1,1]$. Under this condition, we find that
\[
\ix{|f'|^2\;\nu}\ge [d+\alpha\,(d+2)]\,\frac{\nrmx f\qp^2-\nrmx f2^2}{\qp-2}\;.
\]
As a consequence, we also have
\[
\iS{|\nabla u|^2}+\iS{|u|^2}\ge\frac{d+\alpha\,(d+2)}{\qp-2}\(\iS{|u|^\qp}\)^{2/\qp}
\]
for any $u\in\mathrm H^1(\S^d,d\mu)$ such that, using spherical coordinates,
\[
u(\theta,\varphi_1,\varphi_2,...\varphi_{d-1})=u(\pi-\theta,\varphi_1,\varphi_2,...\varphi_{d-1})\quad\forall\;(\theta,\varphi_1,\varphi_2,...\varphi_{d-1})\in[0,\pi]\times[0,2\pi)^{d-1}\;.
\]

\subsection{One more remark} The computation is exactly the same if $\qp\in(1,2)$ and we henceforth also prove the result in such a case. The case $\qp=1$ is the limit case corresponding to the Poincar\'e inequality
\[
\int_{-1}^1|f'|^2\;d\nu_{d+2}\ge d\(\ix{|f|^2}-\left|\ix f\right|^2\)
\]
and arises as a straightforward consequence of the spectral properties of $\L$. The case $\qp=2$ is achieved as a limiting case. It gives rise to the logarithmic Sobolev inequality (see for instance \cite{MR674060}).

\subsection{Limitation of the method} The limitation $\qp\le2^\sharp$ comes from the pointwise condition
\[
h:=|f''|^2+(\qp-1)\,\frac d{d+2}\frac{|f'|^4}{f^2}-\,2\,(\qp-1)\,\frac{d-1}{d+2}\frac{|f'|^2\,f''}f\ge0\;.
\]
Can we find special test functions $f$ such that this quantity can be made negative ? which are admissible, {i.e.} such that $h\,\nu^2$ is integrable ? Notice that at $\qp=2^\sharp$, we have that $f(x)=|x|^{1-d}$ is such that $h\equiv0$, but such a function, or functions obtained by slightly changing the exponent, are not admissible for larger values of $\qp$.

By proving that there is contraction of $\mathcal I$ along the flow, we look for a condition which is stronger than asking that there is contraction of $\mathcal F$ along the flow. It is therefore possible that the limitation $\qp\le2^\sharp$ is intrinsic to the method.

\begin{figure}
\includegraphics[width=6cm]{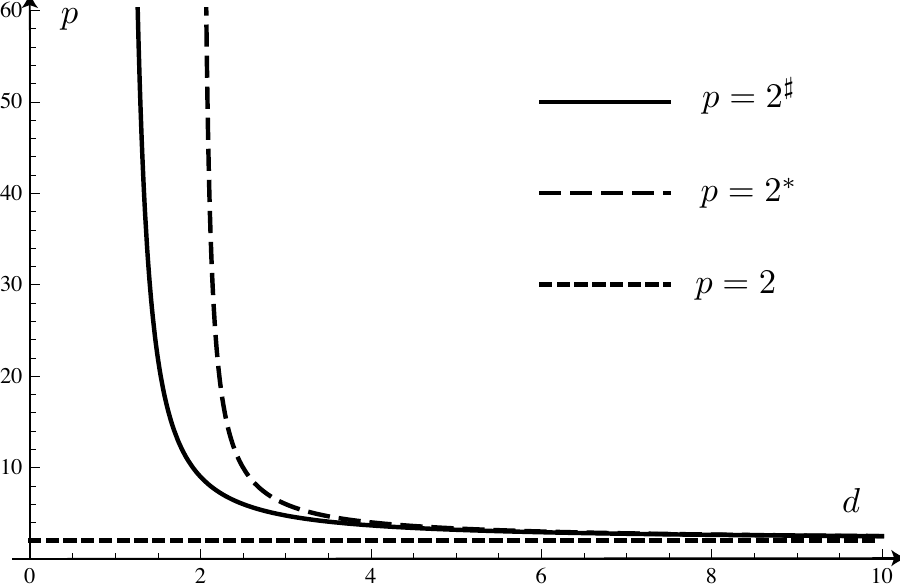}
\caption{\sl Plot of $d\mapsto 2^\sharp=\frac{2\,d^2+1}{(d-1)^2}$ and $d\mapsto 2^*=\frac{2\,d}{d-2}$.}
\end{figure}

\medskip\noindent{\bf Acknowledgements.} J.D.~and M.J.E.~were partially supported by ANR grants \emph{CBDif} and \emph{NoNAP}, and J.D.~by the ECOS project C11E07 \emph{Functional inequalities, asymptotics and dynamics of fronts.} M.K.~was partially supported by Chilean research grants Fondecyt 1090103, Fondo Basal CMM-Chile, Project Anillo ACT-125 CAPDE. M.L.~was supported in part by NSF grant DMS-0901304.\nc \\[4pt]
{\sl\small \copyright~2012 by the authors. This paper may be reproduced, in its entirety, for non-commercial purposes.}
\small

\end{document}